\documentclass[11pt]{article}
\usepackage{amscd}
\usepackage{amsfonts}
\usepackage{amsmath}
\usepackage{amssymb}
\usepackage{amsthm}
\usepackage{bbm}
\usepackage{CJK}
\usepackage{fancyhdr}
\usepackage{graphicx}
\usepackage{indentfirst}
\usepackage{latexsym}
\usepackage{mathrsfs}
\usepackage{xypic}

\newtheorem{theorem}{Theorem}[section]
\newtheorem{lemma}[theorem]{Lemma}
\newtheorem{definition}[theorem]{Definition}
\newtheorem{proposition}[theorem]{Proposition}
\newtheorem{example}[theorem]{Example}

\newtheorem{corollary}[theorem]{Corollary}

\allowdisplaybreaks[4]
\usepackage[top=1in,bottom=1in,left=1.25in,right=1.25in]{geometry}
\date{}
\begin{document}
\renewcommand{\baselinestretch}{1.2}
\renewcommand{\arraystretch}{1.0}
\title{\bf Crossed products and Galois extensions for monoidal Hom-Hopf algebras}

\author{{\bf Yan Ning$$, Daowei Lu\footnote {Corresponding author: Email: ludaowei620@126.com }}\\
{\small Department of Mathematics, Jining University}\\
{\small Qufu, Shandong 273155, P. R. China}
}
\maketitle

\begin{center}
\begin{minipage}{12.cm}

\noindent{\bf Abstract.} Let $(H,\alpha)$ be a monoidal Hom-Hopf algebra, and $(A,\beta)$ a Hom-algebra. In this paper we will introduce the crossed product $(A\#_{\sigma}H,\beta\otimes\alpha)$, which is a Hom-algebra. Then we will introduce the notions of cleft extensions and Galois extensions respectively, and prove that a crossed product is equivalent to a cleft extension and a cleft extension is equivalent to a Galois extension with normal bases property.
 \\

\\

 \end{minipage}
 \end{center}

\section{Introduction}
\label{xxsec1}
The crossed products of Hopf algebras originated in the group theory and were independently introduced in \cite{BCM} and \cite{DT}. Blattner and Montgomery showed in \cite{BM} that a crossed product with invertible 2-cocycle is cleft. In particular crossed products provide examples of Hopf-Galois extensions. Conversely a Hopf-Galois extension with normal bases property is a crossed product as proved in \cite{BM}.

The theory of algebraic deformation has become an important branch of algebras and been well developed recently. The theory has been applied in modules of quantum phenomena, as well as in analysis
of complex systems. Discretization of vector fields via twisted derivations leads to quasi-Hom-Lie structures in which the Jacobi identity is twisted by linear maps (see \cite{HLS,Lar05,Lar07}). The first examples and constructions of quasi-Hom-Lie algebras
and Hom-Lie algebras have been concerned with the $q$-deformations of Witt and Virasoro algebras
obtained when the derivations are replaced by $\sigma$-derivations. In Hom-Lie algebras the Jacobi identity is replaced by the so-called Hom-Jacobi identity, namely
$$[\alpha(x),[y,z]]+[\alpha(y),[z,x]]+[\alpha(z),[x,y]]=0,$$
where $\alpha$ is an endomorphism of the Lie algebra.

Hom-associative algebras was introduced in \cite{MS1} for the first time as an analogue and generalization of associative algebras for Hom-Lie algebras. Here the associativity is replaced by the Hom-associativity and the unit no longer exists, replaced by a weak unit. Dually in \cite{MS2} the notion of Hom-coassociative coalgebra was introduced. Then the concept of Hom-bialgebra and Hom-Hopf algebra were naturally developed as an important generalization of the ordinary Hopf algebras. In \cite{CG} Caenepeel and Goyvaerts illustrated the Hom-structure in the monoidal category approach, and introduced the notion of monoidal Hom-Hopf algebras, which were slightly different from the Hom-Hopf algebras.

Motivated by these ideas, in this paper, firstly we will construct the crossed product of monoidal Hom-Hopf algebras, generalizing the crossed product introduced in \cite{BCM}. Then we will introduce the notions of cleft extensions and Galois extensions of monoidal Hom-Hopf algebras and prove the equivalence among crossed products, cleft extensions and Galois extensions.

This paper is organized as follows: In section 2, we will recall the definitions and results of monoidal Hom-Hopf algebras, such as Hom-algeba,  Hom-coalgebra, Hom-module, and Hom-comodule. In section 3, crossed product in the Hom-setting is constructed, and the relations making the crossed product a Hom-algebra are obtained. In section 4, we will introduce the notion of Hom-analogue cleft extensions and prove that a crossed product is actually a cleft extension and vice versa. In section 5, we will first introduce the definition of Hom-Galois extensions and show that a Galois extension with normal bases property is equivalent to cleft extension, generalizing the result in the classic Hopf algebras.

Throughout this article, all the vector spaces, tensor product and homomorphisms are over a fixed field $k$ unless otherwise stated. We use the Sweedler's notation for the terminologies on coalgebras. For a coalgebra $C$, we write comultiplication $\Delta(c)=\sum c_{1}\otimes c_{2}$ for any $c\in C$.

\section{Preliminary}

In this section, we will recall the basic definitions of monoidal Hom-Hopf algebra from \cite{CG}.

Let $\mathcal{M}_{k}=(\mathcal{M}_{k},\otimes,k,a,l,r)$ be the category of $k$-modules. Now from this category, we could construct a new monoidal category $\mathcal{H}(\mathcal{M}_{k})$. The objects of $\mathcal{H}(\mathcal{M}_{k})$ are pairs $(M,\mu)$, where $M\in\mathcal{M}_{k}$ and $\mu\in Aut_{k}(M)$. Any morphism $f:(M,\mu)\rightarrow (N,\nu)$ in $\mathcal{H}(\mathcal{M}_{k})$ is a $k$-linear map from $M$ to $N$ such that $\nu\circ f=f\circ \mu$. For any objects $(M,\mu)$ and $(N,\nu)$ in $\mathcal{H}(\mathcal{M}_{k})$, the monoidal structure is given by
$$(M,\mu)\otimes(N,\nu)=(M\otimes N,\mu\otimes\nu),$$
and the unit is $(k,id_{k})$.

Generally speaking, all Hom-structure are objects in the monoidal category $\tilde{\mathcal{H}}(\mathcal{M}_{k})=(\mathcal{H}(\mathcal{M}_{k}),\otimes,(k,id_{k}),\tilde{a},\tilde{l},\tilde{r})$ as introduced in \cite{CG}, where the associativity constraint $\tilde{a}$ is given by the formula
$$\tilde{a}_{M,N,L}=a_{M,N,L}\circ((\mu\otimes id)\otimes\lambda^{-1})=(\mu\otimes(id\otimes\lambda^{-1}))\circ a_{M,N,L}$$
for any objects $(M,\mu),\ (N,\nu),\ (L,\lambda)$ in $\mathcal{H}(\mathcal{M}_{k}).$ And the unit constraints $\tilde{l}$ and $\tilde{r}$ are defined by
$$\tilde{l}_{M}=\mu\circ l_{M}=l_{M}\circ(id\otimes\mu),\ \tilde{r}_{M}=\mu\circ r_{M}=r_{M}\circ(\mu\otimes id).$$
The category $\tilde{\mathcal{H}}(\mathcal{M}_{k})$ is called the Hom-category associated to the monoidal category $\mathcal{M}_{k}$. In what follows, we will recall the definitions in \cite{CG} and \cite{LS} on the monoidal Hom-associative algebras, monoidal Hom-coassociative coalgebras, monoidal Hom-modules and monoidal Hom-comodules.

\begin{definition}
 A unital Hom-associative algebra is an object $(A,\alpha)$ in the category $\tilde{\mathcal{H}}(\mathcal{M}_{k})$ together with an element $1_{A}\in A$ and a linear map $m:A\otimes A\rightarrow A,\ a\otimes b\mapsto ab$ such that
\begin{align*}
&\alpha(a)(bc)=(ab)\alpha(c),\ a1_{A}=\alpha(a)=1_{A}a,\\
&\alpha(ab)=\alpha(a)\alpha(b),\ \alpha(1_{A})=1_{A},
\end{align*}
for all $a,b,c\in A.$
\end{definition}

In the setting of Hopf algebras, $m$ is called the Hom-multiplication, $\alpha$ is the twisting automorphism, and $1_{A}$ is the unit.
Let $(A,\alpha)$ and $(A',\alpha')$ be two Hom-algebras. A Hom-algebra map $f:(A,\alpha)\rightarrow(A',\alpha')$ is a linear map such that $f\circ\alpha=\alpha'\circ f$, $f(ab)=f(a)f(b)$ and $f(1_{A})=1_{A}.$

\begin{definition}
 A counital Hom-coassociative coalgebra is an object $(C,\gamma)$ in the category $\tilde{\mathcal{H}}(\mathcal{M}_{k})$ together with linear maps $\Delta:C\rightarrow C\otimes C,\ c\mapsto c_{1}\otimes c_{2}$ and $\varepsilon:C\rightarrow k$ such that
\begin{align*}
&\sum\gamma^{-1}(c_{1})\otimes\Delta(c_{2})=\sum\Delta(c_{1})\otimes\gamma^{-1}(c_{2}),\\
&\sum c_{1}\varepsilon(c_{2})=\sum\varepsilon(c_{1})c_{2}=\lambda^{-1}(c),\\
&\Delta(\gamma(c))=\sum\gamma(c_{1})\otimes\gamma(c_{2}),\ \varepsilon\gamma(c)=\varepsilon(c),
\end{align*}
for all $c\in C.$
\end{definition}

Let $(C,\gamma)$ and $(C',\gamma')$ be two Hom-coalgebras. A Hom-coalgebra map $f:(C,\gamma)\rightarrow(C',\gamma')$ is a linear map such that $f\circ\gamma=\gamma'\circ f,$ $\Delta\circ f=(f\otimes f)\circ\Delta$ and $\varepsilon\circ f=\varepsilon.$

\begin{definition}
A monoidal Hom-bialgebra $H=(H,\alpha,m,1_{H},\Delta,\varepsilon)$ is a bialgebra in the category $\tilde{\mathcal{H}}(\mathcal{M}_{k})$ if $(H,\alpha,m,1_{H})$ is a Hom-algebra and $(H,\alpha,\Delta,\varepsilon)$ is a Hom-coalgebra such that $\Delta$ and $\varepsilon$ are Hom-algebra maps, that is, for any $g,h\in H,$
\begin{align*}
&\Delta(gh)=\Delta(g)\Delta(h),\ \Delta(1_{H})=1_{H}\otimes 1_{H},\\
&\varepsilon(gh)=\varepsilon(g)\varepsilon(h),\ \varepsilon(1_{H})=1.
\end{align*}

A monoidal Hom-bialgebra $(H,\alpha)$ is called a monoidal Hom-Hopf algebra if there exists a linear map $S:H\rightarrow H$(the antipode) such that
$$S\circ\alpha=\alpha\circ S,\ \sum S(h_{1})h_{2}=\varepsilon(h)1_{H}=\sum h_{1}S(h_{2}).$$
\end{definition}

Just as in the case of Hopf algebras, the antipode of monoidal Hom-Hopf algebras is a morphism of Hom-anti-algebras and Hom-anti-coalgebras.

\begin{definition} Let $(A,\alpha)$ be a Hom-algebra. A left $(A,\alpha)$-Hom-module is an object $(M,\mu)$ in $\tilde{\mathcal{H}}(\mathcal{M}_{k})$ together with a linear map $\varphi:A\otimes M\rightarrow M,\ a\otimes m\mapsto am$ such that
$$\alpha(a)(bm)=(ab)\mu(m),\ 1_{A}m=\mu(m),\ \mu(am)=\alpha(a)\mu(m),$$
for all $a,b\in A$ and $m\in M$.
\end{definition}

Similarly we can define the right $(A,\alpha)$-Hom-modules. Let $(M,\mu)$ and $(N,\nu)$ be two left $(A,\alpha)$-Hom-modules, then a linear map $f:M\rightarrow N$ is a called left $A$-module map if $f(am)=af(m)$ for any $a\in A$, $m\in M$ and $f\circ\mu=\nu\circ f$.

\begin{definition}
let $(C,\gamma)$ be a Hom-coalgebra. A right $(C,\gamma)$-Hom-comodule is an object $(M,\mu)$ in $\tilde{\mathcal{H}}(\mathcal{M}_{k})$ together with a linear map $\rho_{M}:M\rightarrow M\otimes C,\ m\mapsto m_{(0)}\otimes m_{(1)}$ such that
\begin{align*}
&\sum\mu^{-1}(m_{(0)})\otimes\Delta(m_{(1)})=\sum\rho_{M}(m_{(0)})\otimes\gamma^{-1}(m_{(1)}),\\ &\sum\varepsilon(m_{(1)})m_{(0)}=\mu^{-1}(m),\\
&\rho_{M}(\mu(m))=\sum\mu(m_{(0)})\otimes\gamma(m_{(1)}),
\end{align*}
for all $m\in M.$
\end{definition}

Let $(M,\mu)$ and $(N,\nu)$ be two right $(C,\gamma)$-Hom-comodules, then a linear map $g:M\rightarrow N$ is a called right $C$-comodule map if $g\circ \mu=\nu\circ g$ and $\rho_{N}(g(m))=(g\otimes id)\rho_{M}(m)$ for any $m\in M.$

\begin{definition} Let $(H,\alpha)$ be a monoidal Hom-Hopf algebra. A Hom-algebra $(B,\beta)$ is called a right $(H,\alpha)$-Hom-comodule algebra if $((B,\beta),\rho)$ is a right $(H,\alpha)$-Hom-comodule and $\rho$ is  a Hom-algebra map.
\end{definition}

\section{Crossed products}
\label{xxsec3}
In this section, we will construct the Hom-crossed product, and make it a Hom-algebra.
\begin{definition}
Let $(H,\alpha)$ be a monoidal Hom-Hopf algebra and $(A,\beta)$  a Hom-algebra. We say $H$ weakly acts on $A$ if there is a $k$-linear map $H\otimes A\rightarrow A$ given by $h\otimes a\mapsto h\cdot a$ such that
$$\beta(h\cdot a)=\alpha(h)\cdot \beta(a),\ h\cdot 1=\varepsilon(h)1,\ h\cdot(ab)=\sum (h_{1}\cdot a_{1})(h_{2}\cdot a_{2}),$$
for any $h\in H$ and $a,b\in A.$
\end{definition}

\begin{definition}
Let $(H,\alpha)$ be a monoidal Hom-Hopf algebra and $(A,\beta)$  a Hom-algebra. Assume that $H$ weakly acts on $A$ and that $\sigma:H\otimes H\rightarrow A$ is convolution invertible. The crossed product $A\#_{\sigma}H$ of $A$ with $H$ is the vector space $A\otimes H$ with the multiplication
$$(a\#h)(b\#k)=\sum a[(\alpha^{-1}(h_{1})\cdot\beta^{-2}(b))\sigma(h_{21},\alpha^{-1}(k_{1}))]\#\alpha^{2}(h_{22})\alpha(k_{2}),$$
for any $h,k\in H$ and $a,b\in A.$
\end{definition}

\begin{proposition}
$(A\#_{\sigma}H,\beta\otimes\alpha)$ is an Hom-associative algebra with identity $1\#1$ if and only if the following conditions hold:
\begin{enumerate}
\item[(1)]
$A$ is a twisted $H$-module, that is $1\cdot a=\beta(a)$ for any $a\in A$, and
$$\sum(\alpha(h_{1})\cdot(l_{1}\cdot \beta^{-1}(a)))\sigma(\alpha(h_{2}),\alpha(l_{2}))=\sum\sigma(\alpha(h_{1}),\alpha(l_{1}))(h_{2}l_{2}\cdot a),\eqno(3.1)$$
for any $h,l\in H$ and $a\in A.$
\item[(2)]
For all $h\in H$, $\sigma(1,h)=\sigma(h,1)=\varepsilon(h)1$, and
$$\sum (\alpha(h_{1})\cdot\sigma(l_{1}, k_{1}))\sigma(\alpha(h_{2}),l_{2}k_{2})=\sum\sigma(\alpha(h_{1}),\alpha(l_{1}))\sigma(h_{2}l_{2},k),\eqno(3.2)$$
for any $h,k,l\in H$.
\item[(3)]
$\sigma$ is a morphism in the category $\tilde{\mathcal{H}}(\mathcal{M}_{k})$, that is, $$\sigma\circ(\alpha\otimes\alpha)=\beta\circ\sigma.$$
\end{enumerate}
\end{proposition}

\begin{proof}
If $(A\#_{\sigma}H,\beta\otimes\alpha)$ is an Hom-associative algebra with identity $1\#1$,
for any $a\#h\in A\#_{\sigma}H$, by the unity,
$$(1\#1)(1\#h)=\sum \beta^{2}(\sigma(1,\alpha^{-1}(h_{1})))\#\alpha^{2}(h_{2})=1\#\alpha(h).$$

Applying $id\otimes\varepsilon_{H}$ to both sides of the above equation, we have $\sigma(1,h)=\varepsilon(h)1.$
Similarly by $(a\#h)(1\#1)=\beta(a)\#\alpha(h)$, we obtain $\sigma(h,1)=\varepsilon(h)1$.

Since $(1\#1)(a\#1)=\sum \beta((1\cdot\beta^{-2}(a))\sigma(1,1))\#1=\beta(a)\#1$, applying $id\otimes\varepsilon_{H}$ to both sides, we have $1\cdot a=\beta(a).$

By $(\beta\otimes\alpha)((a\#h)(b\#k))=(\beta\otimes\alpha)(a\#h)(\beta\otimes\alpha)(b\#k)$, we can get
$$ \sigma\circ(\alpha\otimes\alpha)=\beta\circ\sigma.$$

For any $h,l,k\in H,$ on one hand,
\begin{align*}
&(1\#\alpha(h))[(1\#l)(1\#k)]\\
&=\sum(1\#\alpha(h))(\beta(\sigma(l_{1},k_{1}))\#\alpha(l_{2}k_{2}))\\
                            &=\sum \beta[h_{1}\cdot\sigma(\alpha^{-1}(l_{1}),\alpha^{-1}(k_{1}))\sigma(\alpha(h_{21}),l_{21}k_{21})]\#\alpha^{3}(h_{22})\alpha^{2}(l_{22}k_{22})\\
                            &=\sum (\alpha(h_{1})\cdot\sigma(l_{1}, k_{1}))\sigma(\alpha^{2}(h_{21}),\alpha(l_{21}k_{21}))\#\alpha^{3}(h_{22})\alpha^{2}(l_{22}k_{22}).
\end{align*}
On the other hand,
\begin{align*}
&((1\#h)(1\#l))(1\#\alpha(k))\\
&=\sum(\beta(\sigma(h_{1},l_{1}))\#\alpha(h_{2}l_{2}))(1\#\alpha(k))\\
                            &=\sum\beta(\sigma(h_{1},l_{1}))[\sigma(\alpha(h_{21}l_{21}),\alpha(k_{1}))]\#\alpha^{2}(h_{22}l_{22})\alpha^{2}(k_{2})\\
                            &=\sum\sigma(\alpha(h_{1}),\alpha(l_{1}))\sigma(\alpha(h_{21}l_{21}),\alpha(k_{1}))\#\alpha^{2}(h_{22}l_{22})\alpha^{2}(k_{2}).
\end{align*}
By the Hom-associativity,
\begin{align*}
&\sum (\alpha(h_{1})\cdot\sigma(l_{1}, k_{1}))\sigma(\alpha^{2}(h_{21}),\alpha(l_{21}k_{21}))\#\alpha^{3}(h_{22})\alpha^{2}(l_{22}k_{22})\\
=&\sum\sigma(\alpha(h_{1}),\alpha(l_{1}))\sigma(\alpha(h_{21}l_{21}),\alpha(k_{1}))\#\alpha^{2}(h_{22}l_{22})\alpha^{2}(k_{2}).
\end{align*}
Applying $id\otimes \varepsilon_{H}$ to the above equation, we obtain the relation (3.1).

For any $h,l\in H$ and $a\in A$,
\begin{align*}
&((1\#h)(1\#l))(\beta(a)\#1)\\
&=\sum(\beta(\sigma(h_{1},l_{1}))\#\alpha(h_{2}l_{2}))(\beta(a)\#1)\\
                           &=\sum\beta(\sigma(h_{1},l_{1}))[(h_{21}l_{21}\cdot\beta^{-1}(a))\sigma(\alpha(h_{221}l_{221}),1)]\#\alpha^{3}(h_{222}l_{222})\\
                           &=\sum\beta(\sigma(h_{1},l_{1}))(\alpha(h_{21}l_{21})\cdot a)\#\alpha^{2}(h_{22}l_{22})\\
                           &=\sum\sigma(\alpha(h_{1}),\alpha(l_{1}))(\alpha(h_{21}l_{21})\cdot a)\#\alpha^{2}(h_{22}l_{22}),
\end{align*}
and
\begin{align*}
&(1\#\alpha(h))[(1\#l)(a\#1)]\\
&=\sum (1\#\alpha(h))((\beta^{2}(\alpha^{-1}(l_{1})\cdot \beta^{-2}(a)))\varepsilon(l_{21})\#\alpha^{3}(l_{22})\\
                            &=\sum (1\#\alpha(h))(\alpha(l_{1})\cdot a\#\alpha^{2}(l_{2}))\\
                            &=\sum(\alpha(h_{1})\cdot(l_{1}\cdot \beta^{-1}(a)))\sigma(\alpha^{2}(h_{21}),\alpha^{2}(l_{21}))\#\alpha^{3}(h_{22}l_{22}).
\end{align*}
By the Hom-associativity again, we have
\begin{align*}
&\sum\sigma(\alpha(h_{1}),\alpha(l_{1}))(\alpha(h_{21}l_{21})\cdot a)\#\alpha^{2}(h_{22}l_{22})\\
=&\sum(\alpha(h_{1})\cdot(l_{1}\cdot \beta^{-1}(a)))\sigma(\alpha^{2}(h_{21}),\alpha^{2}(l_{21}))\#\alpha^{3}(h_{22}l_{22}).
\end{align*}
Applying $id\otimes\varepsilon_{H}$ to the above equation, we have the relation (3.1).

Conversely, if the conditions (1), (2) and (3) hold, It is easy to get that $1\#1$ is the unit.
Then for any $a,b,c\in A$, and $h,l,k\in H,$
\begin{align*}
&(\beta(a)\#\alpha(h))((b\#l)(c\#k))\\
&=\sum(\beta(a)\#\alpha(h))(b[(\alpha^{-1}(l_{1})\cdot\beta^{-2}(c))\sigma(l_{21},\alpha^{-1}(k_{1}))]\#\alpha^{2}(l_{22})\alpha(k_{2}))\\
&=\sum \beta(a)[(h_{1}\cdot \beta^{-2}(b)[(\alpha^{-3}(l_{1})\cdot\beta^{-4}(c))\sigma(\alpha^{-2}(l_{21}),\alpha^{-3}(k_{1}))])\sigma(\alpha(h_{21}),\alpha(l_{221})k_{21})]\\
&\quad\#\alpha^{3}(h_{22})(\alpha^{3}(l_{222})\alpha^{2}(k_{22}))\\
&=\sum \beta(a)\{[(h_{11}\cdot \beta^{-2}(b))(h_{12}\cdot[(\alpha^{-3}(l_{1})\cdot\beta^{-4}(c))\sigma(\alpha^{-2}(l_{21}),\alpha^{-3}(k_{1}))])]\\
&\sigma(\alpha(h_{21}),\alpha(l_{221})k_{21})\}\#\alpha^{3}(h_{22})(\alpha^{3}(l_{222})\alpha^{2}(k_{22}))\\
&=\sum \beta(a)\{[(h_{11}\cdot \beta^{-2}(b))[(h_{121}\cdot(\alpha^{-3}(l_{1})\cdot\beta^{-4}(c)))(h_{122}\cdot\sigma(\alpha^{-2}(l_{21}),\alpha^{-3}(k_{1})))]]\\
&\sigma(\alpha(h_{21}),\alpha(l_{221})k_{21})\}\#\alpha^{3}(h_{22})(\alpha^{3}(l_{222})\alpha^{2}(k_{22}))\\
&=\sum \beta(a)\{[(h_{11}\cdot \beta^{-2}(b))(\alpha(h_{121})\cdot(\alpha^{-2}(l_{1})\cdot\beta^{-3}(c)))]\\
&[(\alpha(h_{122})\cdot\sigma(\alpha^{-1}(l_{21}),\alpha^{-2}(k_{1})))\sigma(h_{21},l_{221}\alpha^{-1}(k_{21}))]\}\\
&\#\alpha^{3}(h_{22})(\alpha^{3}(l_{222})\alpha^{2}(k_{22}))\\
&=\sum \beta(a)\{[(\alpha^{-1}(h_{1})\cdot \beta^{-2}(b))(h_{21}\cdot(\alpha^{-2}(l_{1})\cdot\beta^{-3}(c)))][(\alpha^{2}(h_{2211})\cdot\sigma(l_{211},\alpha^{-1}(k_{11})))\\
&\sigma(\alpha^{2}(h_{2212}),l_{212}\alpha^{-1}(k_{12}))]\}\#\alpha^{4}(h_{222})(\alpha^{2}(l_{22})\alpha(k_{2}))\\
&=\sum \beta(a)\{[(\alpha^{-1}(h_{1})\cdot \beta^{-2}(b))(h_{21}\cdot(\alpha^{-2}(l_{1})\cdot\beta^{-3}(c)))][\sigma(\alpha^{2}(h_{2211}),\alpha(l_{211}))\\
&\sigma(\alpha(h_{2212})l_{212},\alpha^{-1}(k_{1}))]\}\#\alpha^{4}(h_{222})(\alpha^{2}(l_{22})\alpha(k_{2}))\ (1)\\
&=\sum \beta(a)\{[(\alpha^{-1}(h_{1})\cdot \beta^{-2}(b))[(\alpha^{-1}(h_{21})\cdot(\alpha^{-3}(l_{1})\cdot\beta^{-4}(c)))\sigma(\alpha(h_{2211}),l_{211})]\\
&\sigma(\alpha^{2}(h_{2212})\alpha(l_{212}),k_{1}]\}\#\alpha^{4}(h_{222})(\alpha^{2}(l_{22})\alpha(k_{2}))\\
&=\sum \beta(a)\{[(\alpha^{-1}(h_{1})\cdot \beta^{-2}(b))[(\alpha(h_{2111})\cdot(\alpha^{-1}(l_{111})\cdot\beta^{-4}(c)))\sigma(\alpha(h_{2112}),l_{112})]\\
&\sigma(\alpha(h_{212})l_{12},k_{1}]\}\#\alpha^{3}(h_{22})(\alpha(l_{2})\alpha(k_{2}))\\
&=\sum \beta(a)\{[(\alpha^{-1}(h_{1})\cdot \beta^{-2}(b))[(\alpha(h_{2111})\cdot(\alpha^{-1}(l_{111})\cdot\beta^{-4}(c)))\sigma(\alpha(h_{2112}),l_{112})]\\
&\sigma(\alpha(h_{212})l_{12},k_{1}]\}\#\alpha^{3}(h_{22})(\alpha(l_{2})\alpha(k_{2}))\\
&=\sum \beta(a)\{[(\alpha^{-1}(h_{1})\cdot \beta^{-2}(b))[\sigma(\alpha(h_{2111}),l_{111})(h_{2112}\alpha^{-1}(l_{112})\cdot\beta^{-3}(c))]\\
&\sigma(\alpha(h_{212})l_{12},k_{1}]\}\#\alpha^{3}(h_{22})(\alpha(l_{2})\alpha(k_{2}))\ (2)\\
&=\sum\{a[(\alpha^{-1}(h_{1})\cdot \beta^{-2}(b))(\sigma(\alpha^{2}(h_{2111}),\alpha(l_{111})))][(\alpha^{2}(h_{2112})\alpha(l_{112})\cdot\beta^{-1}(c))\\
&\sigma(\alpha(h_{212})l_{12},k_{1}]\}\#\alpha^{3}(h_{22})(\alpha(l_{2})\alpha(k_{2}))\\
&=\sum\{a[(\alpha^{-1}(h_{1})\cdot \beta^{-2}(b))\sigma(h_{21},\alpha^{-1}(l_{1}))]\}[(\alpha(h_{221})l_{21}\cdot\beta^{-1}(c))\\
&\sigma(\alpha^{2}(h_{2221})\alpha(l_{221}),k_{1})]\#(\alpha^{4}(h_{2222})\alpha^{3}(l_{222}))\alpha^{2}(k_{2})\\
&=\sum \{a[(\alpha^{-1}(h_{1})\cdot\beta^{-2}(b))\sigma(h_{21},\alpha^{-1}(l_{1}))]\#\alpha^{2}(h_{22})\alpha(l_{2})\}(\beta(c)\#\alpha(k))\\
&=[(a\#h)(b\#l)](\beta(c)\#\alpha(k)).
\end{align*}
The proof is completed.
\end{proof}

\begin{example}
When $\sigma$ is trivial, that is $\sigma(h,k)=\varepsilon(h)\varepsilon(k)1_{A}$, then by (3.1), $A$ is a left $H$-module. Thus $A$ is a left $H$-module algebra. The crossed product $A\#_{\sigma}H$ is reduced to the smash product $A\#H$.
\end{example}

\begin{example}
Let $G$ be a group, and $\alpha$ is an automorphism of $G$. Then $G'=(G,\alpha)$ is a Hom-group with the structure
$$g\cdot h=\alpha(gh),$$
where $\cdot$ means the multiplication in $G'.$

Let $kG$ be the usual group algebra spanned by $G$. Clearly $\alpha$ can be extended to an automorphism of $kG$, still denoted by $\alpha$. Then $(kG,\alpha)$ is a monoidal Hom-Hopf algebra with the structure:
$$\Delta(g)=\alpha^{-1}(g)\otimes\alpha^{-1}(g),\ \varepsilon(g)=1_{k},\ S(g)=g^{-1}.$$
We call $(kG,\alpha)$ a Hom-group algebra. For details we can refer to \cite{CG}.

Assume that $N$ is a normal subgroup of $G$, and $\alpha(N)=N.$ Then $N'=(N,\alpha)$ is also a normal subgroup of $G'$. Hence we have the quotient Hom-group $(G'/N,\overline{\alpha})$, where $\overline{\alpha}$ is an automorphism of $G'/N'$ such that $\overline{\alpha}(\bar{g})=\overline{\alpha(g)}$.
For each coset $\bar{x}\in G/N$,choose a coset representation $\gamma(\bar{x})\in\bar{x}$ satisfying $\alpha\circ\gamma=\gamma\circ\overline{\alpha}$. For simplicity assume $\gamma(\bar{1})=1.$

Now define the left action of $k[G'/N']$ on $kN'$ by
$$\bar{x}\triangleright m=[\gamma(\overline{\alpha}^{-1}(\bar{x}))\cdot\alpha^{-1}(m)]\cdot\gamma(\bar{x})^{-1},$$

and $\sigma:k[G'/N']\otimes k[G'/N']\rightarrow kN'$ by
$$\sigma(\bar{x},\bar{y})=[\gamma(\overline{\alpha}^{-1}(\bar{x}))\cdot\gamma(\overline{\alpha}^{-1}(\bar{y}))]\cdot\gamma(\overline{\alpha}^{-1}(\bar{x})\cdot\overline{\alpha}^{-1}(\bar{y}))^{-1}.$$
It is straightforward to verify the relations in Proposition 3.3 hold. Then we have the crossed product $kG'=kN'\#_{\sigma}k[G'/N']$.
\end{example}

\section{Cleft extensions}

In this section, we will prove that a Hom-crossed product is equivalent to a Hom-cleft extension. This result is a generalization of the theory in the usual Hopf algebras. First we need the following definition.
\begin{definition}
Let $(H,\alpha)$ be a monoidal Hom-Hopf algebra and $(B,\beta)$  a Hom-algebra. Assume $B$ is a right $H$-comodule algebra, and $A=B^{coH}$. Then $A\subset B$ is called $H$-cleft extension if there exists  a right $H$-comodule map $\gamma:H\rightarrow B$ which is convolution inverse.
\end{definition}

Note that we may always assume $\gamma(1)=1.$

\begin{lemma}
Assume that $(B,\beta)$ is a right $H$-Hom-comodule algebra, via $\rho:B\rightarrow B\otimes H$, and that $A\subset B$ is a $H$-cleft extension via $\gamma$ with $\gamma(1)=1.$ Then
\begin{enumerate}
\item[(1)]
$\rho\circ\gamma^{-1}=(\gamma^{-1}\otimes S)\circ\tau\circ\Delta$.
\item[(2)]
$\sum b_{(0)}\gamma^{-1}(b_{(1)})\in A$ for any $b\in B.$
\end{enumerate}
\end{lemma}

\begin{proof}
(1) Since $\rho$ is an algebra map, $\rho\circ\gamma^{-1}$ is the inverse of $\rho\circ\gamma=(\gamma\otimes id)\Delta.$ Let $\lambda=(\gamma^{-1}\otimes S)\circ\tau\circ\Delta$. Then for any $h\in H$,
\begin{align*}
((\rho\circ\gamma)\ast\lambda)(h)&=\sum[(\gamma\otimes id)\Delta(h_{1})][(\gamma^{-1}\otimes S)\circ\tau\circ\Delta(h_{2})]\\
                                 &=\sum[\gamma(h_{11})\otimes h_{12}][\gamma^{-1}(h_{22})\otimes S(h_{21})]\\
                                 &=\sum\gamma(h_{11})\gamma^{-1}(h_{22})\otimes h_{12}S(h_{21})\\
                                 &=\sum\gamma(\alpha^{-1}(h_{1}))\gamma^{-1}(h_{22})\otimes \varepsilon(h_{21})1\\
                                 &=\sum\gamma(\alpha^{-1}(h_{1}))\gamma^{-1}(\alpha^{-1}(h_{2}))\otimes 1\\
                                 &=\varepsilon(h)1\otimes 1.
\end{align*}
Thus $\lambda=\rho\circ\gamma^{-1}$ by the uniqueness of inverses.

(2) For any $b\in B$,
\begin{align*}
\sum\rho(b_{(0)}\gamma^{-1}(b_{(1)}))&=\sum\rho(b_{(0)})\rho(\gamma^{-1}(b_{(1)}))\\
                                     &=\sum(b_{(0)(0)}\otimes b_{(0)(1)})(\gamma^{-1}(b_{(1)2})\otimes S(b_{(1)1}))\ by\ (a)\\
                                     &=\sum b_{(0)(0)}\gamma^{-1}(b_{(1)2})\otimes b_{(0)(1)}S(b_{(1)1})\\
                                     &=\sum \beta^{-1}(b_{(0)})\gamma^{-1}(\alpha(b_{(1)22}))\otimes b_{(1)1}S(\alpha(b_{(1)21}))\\
                                     &=\sum \beta^{-1}(b_{(0)})\gamma^{-1}(b_{(1)2})\otimes \alpha(b_{(1)11})S(\alpha(b_{(1)12}))\\
                                     &=\sum \beta^{-1}(b_{(0)})\gamma^{-1}(\alpha^{-1}(b_{(1)}))\otimes 1\\                                   &=\sum\beta^{-1}(b_{(0)}\gamma^{-1}(b_{(1)})\otimes 1.
\end{align*}
Hence we have (2).
\end{proof}

\begin{proposition}
Let $A\subset B$ be right $H$-cleft extension via $\gamma:H\rightarrow B$. Then there is a crossed product action of $H$ on $A$ given by
$$h\cdot a=\sum(\gamma(h_{1})\beta^{-1}(a))\gamma^{-1}(\alpha(h_{2})),$$
and a convolution inverse map $\sigma:H\otimes H\rightarrow A$ given by
$$\sigma(h,k)=\sum(\gamma(h_{1})\gamma(k_{1}))\gamma^{-1}(h_{2}k_{2}).$$
Then we have the crossed product $A\#_{\sigma}H$. Moreover $\Phi:A\#_{\sigma}H\rightarrow B,\ a\#h\mapsto a\gamma(h)$ is a Hom-algebra map. Moreover $\Phi$ is both a left $A$-module and right $H$-comodule map, where $a\cdot(b\#h)=\beta^{-1}(a)b\#\alpha(h)$ and $\sum(a\#h)_{(0)}\otimes(a\#h)_{(1)}=\sum\beta^{-1}(a)\#h_{1}\otimes \alpha(h_{2})$.
\end{proposition}

\begin{proof}
We first show that $h\cdot a\in A$ for any $a\in A,h\in H$. Indeed
\begin{align*}
\rho(h\cdot a)&=\rho(\sum (\gamma(h_{1})\beta^{-1}(a))\gamma^{-1}(\alpha(h_{2}))\\
              &=\sum[\rho\gamma(h_{1})\rho(\beta^{-1}(a))]\rho\gamma^{-1}(\alpha(h_{2}))\\
              &=\sum[(\gamma(h_{11})\otimes h_{12})(\beta^{-2}(a)\otimes 1)](\beta(\gamma^{-1}(h_{2})_{(0)})\otimes\alpha(\gamma^{-1}(h_{2})_{(1)}))\\
              &=\sum(\gamma(h_{11})\beta^{-2}(a)\otimes \alpha(h_{12}))(\beta(\gamma^{-1}(h_{22}))\otimes\alpha(S(h_{21})))\\
              &=\sum(\gamma(h_{11})\beta^{-2}(a))\beta(\gamma^{-1}(h_{22}))\otimes\alpha(h_{12})\alpha(S(h_{21}))\\
              &=\sum(\gamma(\alpha^{-1}(h_{1}))\beta^{-2}(a))\beta(\gamma^{-1}(h_{22}))\otimes\varepsilon(h_{21})1\\
              &=\sum(\gamma(\alpha^{-1}(h_{1}))\beta^{-2}(a))\gamma^{-1}(h_{2})\otimes1\\
              &=\beta^{-1}(h\cdot a)\otimes1.
\end{align*}

Thus $h\cdot a\in B^{coH}=A.$ It is straightforward to verify that $H$ weakly acts on $A$.

And for any $h,k\in H $, we have
\begin{align*}
\rho(\sigma(h,k))&=\sum (\rho\gamma(h_{1})\rho\gamma(k_{1}))\rho\gamma^{-1}(h_{2}k_{2})\\
                 &=\sum[(\gamma(h_{11})\otimes h_{12}))(\gamma(k_{11})\otimes k_{12})][\gamma^{-1}(h_{22}k_{22})\otimes S(h_{21}k_{21})]\\
                 &=\sum(\gamma(h_{11})\gamma(k_{11}))\gamma^{-1}(h_{22}k_{22})\otimes (h_{12}k_{12})S(h_{21}k_{21})\\
                 &=\sum(\gamma(\alpha^{-1}(h_{1}))\gamma(\alpha^{-1}(k_{1})))\gamma^{-1}(h_{22}k_{22})\otimes \varepsilon(h_{21}k_{21})1\\
                 &=\sum(\gamma(\alpha^{-1}(h_{1}))\gamma(\alpha^{-1}(k_{1})))\gamma^{-1}(\alpha^{-1}(h_{2}k_{2}))\otimes 1\\
                 &=\beta^{-1}(\sigma(h,k))\otimes 1.
\end{align*}

Thus $\sigma(h,k)\in A.$

Define $\Psi:B\rightarrow A\#_{\sigma}H$ by $b\mapsto\sum b_{(0)(0)}\gamma^{-1}(b_{(0)(1)})\#b_{(1)}$. By Lemma 3.2(b), we know that it makes sense. Next we will show that $\Phi$ and $\Psi$ are mutual inverses. First for any $b\in B$,
\begin{align*}
\Phi\Psi(b)&=\sum\Phi(b_{(0)(0)}\gamma^{-1}(b_{(0)(1)})\#b_{(1)})=\sum(b_{(0)(0)}\gamma^{-1}(b_{(0)(1)}))\gamma(b_{(1)})\\
           &=\sum\beta(b_{(0)(0)})(\gamma^{-1}(b_{(0)(1)})\gamma(\alpha^{-1}(b_{(1)})))=\sum b_{(0)}(\gamma^{-1}(b_{(1)1})\gamma(b_{(1)2}))\\
           &=\sum b_{(0)}\varepsilon(b_{(1)})1=b,
\end{align*}

and for any $a\#h\in A\#_{\sigma}H$,
\begin{align*}
\Psi\Phi(a\#h)&=\Psi(a\gamma(h))=\sum (a_{(0)(0)}\gamma(h)_{(0)(0)})\gamma^{-1}(a_{(0)(1)}\gamma(h)_{(0)(1)})\#a_{(1)}\gamma(h)_{(1)}\\
              &=\sum (\beta^{-2}(a)\gamma(h)_{(0)(0)})\gamma^{-1}(\alpha(\gamma(h)_{(0)(1)}))\#\alpha(\gamma(h)_{(1)})\\
              &=\sum (\beta^{-2}(a)\gamma(h_{1})_{(0)})\gamma^{-1}(\alpha(\gamma(h_{1})_{(1)}))\#\alpha(h_{2})\\
              &=\sum (\beta^{-2}(a)\gamma(h_{11}))\gamma^{-1}(\alpha(h_{12}))\#\alpha(h_{2})\\
              &=\sum \beta^{-1}(a)(\gamma(h_{11})\gamma^{-1}(h_{12}))\#\alpha(h_{2})\\
              &=a\#h.
\end{align*}

Thus $\Psi=\Phi^{-1}$. Moreover $\Phi$ is a Hom-algebra map. Indeed, firstly $\Phi\circ(\beta\otimes\alpha)=\beta\circ\Phi$, and
\begin{align*}
&\Phi((a\#h)(b\#k))\\
=&\sum \{a[(\alpha^{-1}(h_{1})\cdot\beta^{-2}(b))\sigma(h_{21},\alpha^{-1}(k_{1}))]\}\gamma(\alpha^{2}(h_{22})\alpha(k_{2}))\\
=&\sum \{a[((\gamma(\alpha^{-1}(h_{11}))\beta^{-3}(b))\gamma^{-1}(h_{12}))((\gamma(h_{211})\gamma(\alpha^{-1}(k_{11})))\gamma^{-1}(h_{212}\alpha^{-1}(k_{12})))]\}\\
&\gamma(\alpha^{2}(h_{22})\alpha(k_{2}))\\
=&\sum \{a[[(\gamma(\alpha^{-1}(h_{11}))\beta^{-3}(b))(\gamma^{-1}(\alpha^{-1}(h_{12}))\gamma(h_{211}))](\gamma(k_{11})(\gamma^{-1}(h_{212}\alpha^{-1}(k_{12}))))]\}\\
&\gamma(\alpha^{2}(h_{22})\alpha(k_{2}))\\
=&\sum \{a[(\gamma(\alpha^{-1}(h_{1}))\beta^{-2}(b))(\gamma(k_{11})(\gamma^{-1}(\alpha^{-1}(h_{21})\alpha^{-1}(k_{12}))))]\}
\gamma(\alpha^{2}(h_{22})\alpha(k_{2}))\\
=&\sum(a(\gamma(h_{1})\beta^{-1}(b)))[(\gamma(\alpha(k_{11}))\gamma^{-1}(h_{21}k_{12}))\gamma(\alpha(h_{22})k_{2})]\\
=&\sum(a(\gamma(h_{1})\beta^{-1}(b)))\gamma(\alpha(k))\varepsilon(h_{2})\\
=&\Phi(a\#h)\Phi(b\#k).
\end{align*}

Thus $A\#_{\sigma}H\cong B$ as Hom-algebra. The conditions (3.1) and (3.2) hold. Finally it is easy to check that $\Phi$ is a left $A$-module and right $H$-comodule map. The proof is completed.
\end{proof}

\begin{proposition}
Let $A\#_{\sigma}H$ be a crossed product, and define the map $\gamma:H\rightarrow A\#_{\sigma}H$ by $\gamma(h)=1\#\alpha^{-1}(h)$. Then $\gamma$ is convolution invertible with inverse
$$\gamma^{-1}(h)=\sum \sigma^{-1}(S(h_{21}),h_{22})\#S(h_{1}).$$
\end{proposition}

\begin{proof}
Set $\mu(h)=\sum \sigma^{-1}(S(h_{21}),h_{22})\#S(h_{1})$. Then
\begin{align*}
(\mu\ast\gamma)(h)&=\sum(\sigma^{-1}(S(h_{121}),h_{122})\#S(h_{11}))(1\#\alpha^{-1}(h_{2}))\\
                  &=\sum\sigma^{-1}(S(h_{121}),h_{122})\sigma(S(h_{112}),\alpha^{-1}(h_{21}))\#S(\alpha(h_{111}))h_{22}\\
                  &=\sum\sigma^{-1}(S(h_{212}),\alpha(h_{2211}))\sigma(S(h_{211}),\alpha(h_{2212}))\#S(\alpha^{-1}(h_{1}))\alpha(h_{222})\\
                  &=\sum 1\#\varepsilon(h_{21})\varepsilon(h_{221})S(\alpha^{-1}(h_{1}))\alpha(h_{222})\\
                  &=\varepsilon(h)1\#1.
\end{align*}

Thus $\mu$ is the left inverse of $\gamma$. And it is straightforward to verify that $\mu$ is the right inverse of $\gamma$, as done in Proposition 7.2.7 in \cite{Mon}.

Obviously $\gamma$ is  a right $H$-comodule map. This completes the proof.
\end{proof}

By the above two propositions, we obtain the following theorem directly.

\begin{theorem}
$A\subset B$ is a $H$-cleft extension if and only if $B\cong A\#_{\sigma}H.$
\end{theorem}

\section{Galois extension}

In this section, we will introduce the Galois extensions in the Hom situation, and generalize the result in the Hopf algebra setting that a cleft extension is equivalent to a Galois extension with normal bases property.

Let $(A,\beta)$ be a Hom-algebra, and $(M,\mu)$, $(N,\lambda)$ be a right and left Hom-monoidal module respectively. Define a subspace $X$ of $M\otimes N$ by $X=\{m\cdot a\otimes n-\mu(m)\otimes a\cdot \lambda^{-1}(n)|\forall m\in M,n\in N,a\in A\}$. Set $M\otimes_{A}N$ be the quotient $(M\otimes N)/X$. That is, in $M\otimes_{A}N$, $m\cdot a\otimes n=\mu(m)\otimes a\cdot \lambda^{-1}(n)$.

\begin{definition}
Let $(H,\alpha)$ be a monoidal Hom-Hopf algebra and $(A,\beta)$ be a Hom-algebra. Assume $A$ is a right $H$-comodule algebra. Then the extension $A^{coH}$ is called a right $H$-Galois if the map $\varphi:A\otimes_{A^{coH}}A\rightarrow A\otimes H$ given by $\varphi(a\otimes b)=\beta^{-1}(a)b_{(0)}\otimes \alpha(b_{(1)})$ is a bijective in the category $\mathcal{M}_{k}$.
\end{definition}

\begin{definition}
Let $A\subset B$ be a right $H$-extension. The extension has the normal base property if $B\cong A\otimes H$ as left $A$-module and right $H$-comodule, where the left $A$-module action and right $H$-comodule coaction on $A\otimes H$ is defined in Proposition 4.3.
\end{definition}

\begin{theorem}
Let $A\subset B$ be a right $H$-extension. Then the following are equivalent
\begin{enumerate}
\item[(1)]
$A\subset B$ is $H$-cleft,
\item[(2)]
$A\subset B$ is $H$-Galois and has the normal base property.
\end{enumerate}
\end{theorem}

\begin{proof}
$(1)\Rightarrow (2)$. By Proposition 2.3, if $A\subset B$ is $H$-cleft, $B\backsimeq A\#_{\sigma} H$ as left $A$-module and right $H$-comodule. Thus the normal base property is satisfied.

Assume $\varphi:B\otimes_{A}B\rightarrow B\otimes H$ is given by $\varphi(a\otimes b)=\beta^{-1}(a)b_{(0)}\otimes \alpha(b_{(1)})$.  Let $\gamma:H\rightarrow B$ be the map such that $A\subset B$ is $H$-cleft. Define $\psi:B\otimes H \rightarrow B\otimes_{A}B$ by
$$\psi(b\otimes h)=\sum \beta^{-1}(b)\gamma^{-1}(h_{1})\otimes \gamma(\alpha(h_{2})).$$
Then on one hand,
\begin{align*}
\varphi\psi(b\otimes h)&=\varphi(\sum\beta^{-1}(b)\gamma^{-1}(h_{1})\otimes \gamma(\alpha(h_{2})))\\
                       &=\sum (\beta^{-2}(b)\gamma^{-1}(\alpha^{-1}(h_{1})))\gamma(\alpha(h_{2}))_{(0)}\otimes\alpha(\gamma(\alpha(h_{2}))_{(1)})\\
                       &=\sum (\beta^{-2}(b)\gamma^{-1}(\alpha^{-1}(h_{1})))\gamma(\alpha(h_{21}))\otimes\alpha^{2}(h_{22})\\
                       &=\sum \beta^{-1}(b)(\gamma^{-1}(h_{11})\gamma(h_{21}))\otimes\alpha(h_{2})\\
                       &=b\otimes h.
\end{align*}
On the other hand,
\begin{align*}
\psi\varphi(a\otimes b)&=\psi(\sum \beta^{-1}(a)b_{(0)}\otimes\alpha(b_{(1)}))\\
                       &=\sum(\beta^{-2}(a)\beta^{-1}(b_{(0)}))\gamma^{-1}(\alpha(b_{(1)1}))\otimes\gamma(\alpha^{2}(b_{(1)2}))\\
                       &=\sum\beta^{-1}(a)(\beta^{-1}(b_{(0)})\gamma^{-1}(b_{(1)1}))\otimes\gamma(\alpha^{2}(b_{(1)2}))\\
                       &=\sum\beta^{-1}(a)(b_{(0)(0)}\gamma^{-1}(b_{(0)(1)}))\otimes\gamma(\alpha(b_{(1)}))\\
                       &\ \ \ \ Since\ \Sigma b_{(0)}\gamma^{-1}(b_{(1)})\in A\\
                       &=\sum a\otimes(b_{(0)(0)}\gamma^{-1}(b_{(0)(1)}))\gamma(b_{(1)})\\
                       &=a\otimes b.
\end{align*}
Now $\varphi$ is a bijective. Thus $A\subset B$ is $H$-Galois.

$(2)\Rightarrow(1)$. Assume $A\subset B$ is $H$-Galois with $\varphi$ being bijective and there exists a bijective $\theta:A\otimes H\rightarrow B$ which is a left $A$-module and right $H$-comodule map. Define $\gamma:H\rightarrow B$ by $\gamma(h)=\theta(1\otimes\alpha^{-1}(h))$. Then it is easy to see that $\gamma$ is a right $H$-comodule map.

Define $g\in Hom_{A}(B,A)$ by $g=\beta\circ(id\otimes\varepsilon)\circ\theta^{-1}$. Obviously $g(\gamma(h))=\varepsilon(h)1$. Then we define $\mu:H\rightarrow B$ by
$$\mu(h)=m(id\otimes g)\varphi^{-1}(1\otimes\alpha^{-1}(h)).$$

We claim that $\mu=\gamma^{-1}$. Indeed firstly for any $h\in H$,
\begin{align*}
(\gamma\ast\mu)(h)&=\sum\gamma(h_{1})m(id\otimes g)\varphi^{-1}(1\otimes\alpha^{-1}(h_{2}))\\
                  &=\sum m(id\otimes g)[(\gamma(\alpha^{-1}(h_{1}))\otimes 1)\varphi^{-1}(\gamma(h_{1})\otimes\alpha^{-1}(h_{2}))]\\
                  &=\sum m(id\otimes g)\varphi^{-1}(\gamma(h_{1})\otimes h_{2})\\
                  &\ \ \ \ \ since\ \varphi^{-1}(\beta(b)\otimes\alpha(h))=(b\otimes1)\varphi^{-1}(1\otimes h)\\
                  &=m(id\otimes g)(1\otimes\gamma\alpha^{-1}(h))\\
                  &=\varepsilon(h)1.
\end{align*}

Secondly, note that
$$(\varphi\otimes id)(\beta^{-1}\otimes id\otimes\alpha)(id\otimes\rho)=(\beta^{-1}\otimes id\otimes\alpha)(id\otimes\Delta)\varphi,$$

hence we have
$$(\beta^{-1}\otimes id\otimes\alpha)(id\otimes\rho)\varphi^{-1}=(\varphi^{-1}\otimes id)(\beta^{-1}\otimes id\otimes\alpha)(id\otimes\Delta).$$

Then for any $h\in H,$ denote $\varphi^{-1}(1\otimes \alpha^{-1}(h))=\sum a_{i}\otimes b_{i}$. We have
$$\sum\beta^{-1}(a_{i})\otimes b_{i(0)}\otimes b_{i(1)}=\sum\varphi^{-1}(1\otimes \alpha^{-1}(h_{1}))\otimes \alpha^{-1}(h_{2}).$$

Also since $\theta$ is a right $H$-comodule map,

\begin{align*}
(\beta^{-1}\otimes id)\theta^{-1}(b)&=(id\otimes\varepsilon\otimes id)(\sum\theta^{-1}(b)_{(0)}\otimes\theta^{-1}(b)_{(1)})\\
                                    &=(id\otimes\varepsilon\otimes id)(\sum\theta^{-1}(b_{(0)})\otimes b_{(1)})\\
                                    &=\sum g(\beta^{-1}(b_{(0)}))\otimes b_{(1)}.
\end{align*}

By $\gamma(h)=\theta(1\otimes\alpha^{-1}(h))$, we have $b=\sum g(b_{(0)})\gamma(b_{(1)})$. Then
\begin{align*}
(\mu\ast\gamma)(h)&=\sum[m(id\otimes g)\varphi^{-1}(1\otimes\alpha^{-1}(h_{1}))]\gamma(h_{2})\\
                  &=\sum[m(id\otimes g)(\beta^{-1}(a_{i})\otimes b_{i(0)})]\gamma(\alpha(b_{i(1)}))\\
                  &=\sum a_{i}(g(b_{i(0)})\gamma(b_{i(1)}))\\
                  &=\sum a_{i}b_{i}=\varepsilon(h)1.
\end{align*}

Thus $\gamma$ is convolution inverse.
The proof is completed.
\end{proof}

By Theorem 2.5, we have the following result.

\begin{corollary}
Let $A\subset B$ be a right $H$-extension. Then $A\subset B$ is Galois with the normal bases property if and only if $B\cong A\#_{\sigma} H$.
\end{corollary}

\section*{Acknowledgements}

This work was supported by the NSF of China (No. 11901240) and the NSF of Shandong Province (No. ZR2018PA006).


\begin{thebibliography}{10}



\bibitem{BCM} R. Blattner, M. Cohen, S. Montgomery. Crossed products and inner actions
of Hopf algebras. Trans. Am. Math. Soc., 298(1986): 671--711.

\bibitem{BM} R. Blattner, S. Montgomery. Crossed products and Galois extensions of Hopf
algebras. Pac. J. Math., 137(1989): 37--54.

\bibitem{CG} S. Caenepeel, I. Goyvaerts. Monoidal Hom-Hopf algebras. Commu. Alg., 39(2011): 2216--2240.

\bibitem{Chen} Y. Y. Chen, Y. Wang, L. Y. Zhang. The construction of Hom-Lie bialgebras. J. Lie Theory 22(2010), 1075--1089.


\bibitem{DT} Y. Doi, M. Takeuchi. Cleft comodule algebras for a bialgebra. Commu.
Alg., 14(1986): 801--817.

\bibitem{HLS} J. T. Hartwig, D. Larsson, S. D. Silvestrov. Deformations of Lie algebras
using $\sigma$-derivations. J. Algebra,  295(2006): 314--361.

\bibitem{K} C. Kassel.  Quantum Groups. Graduate Texts in Math. Vol. 155, 1995.

\bibitem{Lar05} D. Larsson, S. D. Silvestrov. Quasi-Hom-Lie algebras, central extensions and 2-cocycle-like identities. J. Algebra, 288(2005): 321--344.

\bibitem{Lar07} D. Larsson, S. D. Silvestrov. Quasi-deformations of$sl_2(f)$ using twisted derivations, Commun. Algebra, 35(2007): 4303--4318.


\bibitem{LS}
L. Liu, B. Shen. Radford's biproducts and Yetter-Drinfeld modules for monoidal Hom-Hopf algebras. J. Math. Phys., 55(2014), 031701.

\bibitem{Mon} S. Montgomery. Hopf algebras and their actions on rings. CBNS series in Math., Vol. 82, Am. Math. Soc., Provience, 1993.


\bibitem{MS1} A. Makhlouf, S. D. Silvestrov. Hom-algebra structure. J. Gen. Lie Theory Appl., 2(2008): 52--64.

\bibitem{MS2} A. Makhlouf, S. D. Silvestrov. Hom-algebras and Hom-coalgebras. J. Alg. Appl., 9(2010): 553--589.

\bibitem{Yau} D. Yau. Hom-quantum group I: Quasi-triangular Hom-bialgebras. J. Phys. A: Math. Theor. 45(2012), 065203.
\end{thebibliography}
\end{document}